\documentclass[12pt]{article}

\usepackage{amssymb,amsmath,amsthm,mathabx}
\usepackage{graphicx, epstopdf, caption, subcaption}

      \theoremstyle{plain}
      \newtheorem{theorem}{Theorem}[section]
      \newtheorem{lemma}{Lemma}[section]

      \theoremstyle{definition}
      \newtheorem{definition}{Definition}[section]
      
      \theoremstyle{remark}
      \newtheorem{remark}{Remark}[section]

\def\st{\Asterisk}

\begin{document}

\title{Checkerboard embeddings of $\st$-graphs into nonorientable surfaces}
\author{Tyler Friesen
\footnote{Ohio State University, 154 W 12th Ave., Columbus, Ohio 43210}
\footnote{\tt {friesen.15@osu.edu}}
\and Vassily Olegovich Manturov
\footnote{Peoples' Friendship University of Russia, Moscow 117198,
Ordjonikidze St., 3} \footnote{\tt {vomanturov@yandex.ru}}
\footnote{Partially supported by grants of the Russian Government
11.G34.31.0053, RF President NSh � 1410.2012.1, Ministry of
Education and Science of the Russian Federation 14.740.11.0794.}}

\maketitle

\begin{abstract}
This paper considers $\st$-graphs in which all vertices have degree 4 or 6, and studies the question of calculating the genus of nonorientable surfaces into which such graphs may be embedded. In a previous paper \cite{firstpaper} by the authors, the problem of calculating whether a given $\st$-graph in which all vertices have degree 4 or 6 admits a $\mathbb Z_2$-homologically trivial embedding into a given orientable surface was shown to be equivalent to a problem on matrices. Here we extend those results to nonorientable surfaces. The embeddability condition that we obtain yields quadratic-time algorithms to determine whether a $\st$-graph with all vertices of degree 4 or 6 admits a $\mathbb Z_2$-homologically trivial embedding into the projective plane or into the Klein bottle.
\end{abstract}

{\bf Keywords:} Graph, $\st$-graph, surface, embedding, genus

{\bf AMS Subject Classification:} Primary 05C10; Secondary 57C15, 57C27

\section{Introduction}

\begin{definition}
A \emph{$\st$-graph} is a graph which at each vertex has a bijection from the outgoing half-edges to the vertices of a cycle graph. Half-edges which are mapped to adjacent vertices are (formally) \emph{adjacent}. Half-edges are said to be \emph{opposite} if they are mapped to vertices of maximal distance in the cycle graph.
\end{definition}

\begin{remark}
By an embedding of a $\st$-graph $\Gamma$ into a surface $S$ we always mean an embedding of  $\Gamma$ into $S$ such that the formal relation of being adjacent coincides with the relation of being adjacent induced by the embedding.
\end{remark}

\begin{definition}
An \emph{angle} in a $\st$-graph is a pair of adjacent half-edges at a vertex.
\end{definition}

\begin{definition}
A \emph{checkerboard embedding} of a  $\st$-graph $\Gamma$ into $S$ is an embedding such that the cells of $S \setminus \Gamma$ admit a 2-coloring under which any cells with a common edge have different colors. 
\end{definition}

\begin{remark}
Checkerboard embeddings are exactly those embeddings whose first $\mathbb Z_2$-homology class is zero.
\end{remark}

In \cite{manturov1} the second named author (V.O.M.) gave a solution to the question of whether four-valent framed graphs are planar. In \cite{manturov2}, he addressed the question of determining the genus of surfaces into which four-valent framed graphs can be embedded, in particular considering the special case of surfaces into which four-valent framed graphs may be checkerboard-embedded. In \cite{friesen}, the first named author (T.F.) introduced $\st$-graphs as a generalization of four-valent framed graphs, and gave a planarity condition for $\st$-graphs with each vertex of degree 4 or 6. In \cite{firstpaper}, the authors characterized the genera of orientable surfaces into which $\st$-graphs with each vertex of degree 4 or 6 may be checkerboard-embedded, generalizing some of the results of \cite{manturov1}. In this paper, we continue the project of generalizing results about embeddability properties of framed four-valent graphs to $\st$-graphs with each vertex of degree 4 or 6, now considering checkerboard-embeddability into nonorientable surfaces. In Theorem \ref{main} we show that this is equivalent to a problem on matrices. Our methods are close to those used in our previous paper \cite{firstpaper}, which were themselves based closely on those used by the second named author (V.O.M.) in \cite{manturov2} for framed four-valent graphs.

The goal of this paper is to provide a method for determining whether a given $\st$-graph $\Gamma$ has a checkerboard embedding into a nonorientable surface of genus $g$. We show that this is equivalent to a problem on matrices. To accomplish this, we fix a cycle $C$ in $\Gamma$ satisfying certain properties, called a rotating-splitting cycle. Then we define a correspondence between checkerboard embeddings of $\Gamma$ and ``permissible separations'' of a chord diagram $D'_{\Gamma, C}$, where the result of a permissible separation is a pair of chord diagrams. We then show that the number of white (black) cells in the embedding is equal to the number of circles resulting from surgery of the first (second) of these chord diagrams. The Circuit-Nullity Theorem allows us to calculate the number of circles resulting from surgery of each chord diagram in terms of the rank of their intersection matrices over $\mathbb Z_2$. From this we have the total number of cells in the embedding, from which the genus $g$ of the surface can be easily calculated.

The authors of this paper would like to thank Victor Anatolievich Vassiliev and Sergei Vladimirovich Chmutov for valuable discussions.

\section{Basic Notions}

\begin{definition}
A \emph{$\st$-atom} is a closed 2-surface $S$ into which a connected graph $\Gamma$ (the \emph{skeleton} of the $\st$-atom) is embedded in such a way that it divides $S$ into black and white cells so that cells sharing an edge have different colors.
\end{definition}

This embedding induces a $\st$-structure on the skeleton. The $\st$-structure at each vertex determines a set of $d$ angles among which we say that two angles are adjacent if they share a half-edge. Two adjacent angles never have the same color. Thus the angles around a vertex can be partitioned into two sets $A_1$ and $A_2$ in such a way that for any $\st$-atom corresponding to the $\st$-graph $\Gamma$, either all angles in $A_1$ are black and all angles in $A_2$ are white, or all angles in $A_1$ are white and all angles in $A_2$ are black. Thus given a connected $\st$-graph $\Gamma$, the $\st$-atoms corresponding to $\Gamma$ are uniquely determined by a choice of one of the two possible colorings at each vertex. Thus the main problem can be reformulated as follows:

Given a $\st$-graph $\Gamma$ in which all vertices have degree 4 or 6, choose a coloring for the angles around each vertex such that the genus of the resulting atom is minimal. If such a graph has $n$ vertices, there are $2^n$ corresponding $\st$-atoms.

%\begin{definition}
%A  $\st$-graph satisfies the \emph{source-sink condition} if each edge of it can be endowed with an orientation in such a way that for any two adjacent edges at a vertex, one edge is emanating and the other is incoming.
%\end{definition}

%\begin{lemma}\cite{firstpaper}
%A $\st$-graph satisfies the source-sink condition if and only if it is checkerboard-embeddable into an orientable surface. Furthermore, if a $\st$-graph satisfies the source-sink condition, every surface into which it is checkerboard-embeddable is orientable.
%\end{lemma}

%\begin{proof}
%Suppose that a $\st$-graph $\Gamma$ is checkerboard-embedded into an orientable surface $S$. We can orient the edges in such a way that the edges around any white cell are oriented clockwise around that cell, showing that $\Gamma$ meets the source-sink condition. Conversely, if a $\st$-graph $\Gamma$ is checkerboard-embedded into a surface $S$, we can orient the cells of the embedding in such a way that in white cells, the edges around the cell are oriented clockwise, and in black cells, the edges around the cell are oriented counterclockwise.
%\end{proof}

\begin{definition}
An \emph{Euler circuit} $C$ of a $\st$-graph $\Gamma$ is a surjective mapping $S^1 \to \Gamma$ which is one-to-one except at the vertices, and such that every vertex of degree $d$ has $\frac{d}{2}$ preimages.
\end{definition}

\begin{definition}
Given an Euler circuit $C$ of a $\st$-graph $\Gamma$, a 4-vertex $v \in \Gamma$ is \emph{rotating} with respect to $C$ if for every $a \in C^{-1}(\{v\})$, $C(a+\epsilon)$ and $C(a-\epsilon)$ are on adjacent half-edges around $v$.
\end{definition}

\begin{definition}
Given an Euler circuit $C$ of a $\st$-graph $\Gamma$, a 6-vertex $v \in \Gamma$ is \emph{rotating} with respect to $C$ if for every $a \in C^{-1}(\{v\})$, $C(a+\epsilon)$ and $C(a-\epsilon)$ are on adjacent half-edges around $v$.
\end{definition}

\begin{definition}
Given an Euler circuit $C$ of a $\st$-graph $\Gamma$, a 6-vertex $v \in \Gamma$ is \emph{splitting} with respect to $C$ if for some $a \in C^{-1}(\{v\})$, $C(a+\epsilon)$ and $C(a-\epsilon)$ are on opposite half-edges around $v$, and for the other two points $b, c \in C^{-1}(\{v\})$, $C(b+\epsilon)$ and $C(b-\epsilon)$ are on adjacent half-edges, and $C(c+\epsilon)$ and $C(c-\epsilon)$ are on adjacent half-edges.
\end{definition}

\begin{definition}
A \emph{rotating-splitting circuit} is a circuit with respect to which every vertex is rotating or splitting.
\end{definition}

\begin{definition}
A rotating-splitting circuit induces an orientation on the half-edges around a rotating 6-vertex. If the order of the edges containing $C(a+\epsilon)$ and $C(a-\epsilon)$ agrees with this orientation, then the angle is said be \emph{untwisted}; otherwise it is \emph{twisted}. See Figures \ref{r0t}, \ref{r1t}, \ref{r2t}, \ref{r3t}.
\end{definition}

\begin{figure}
\centering
\includegraphics[trim = 0cm 8cm 0cm 6cm, clip=true, totalheight=4cm]{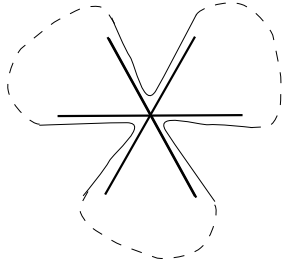}
\caption{A rotating vertex with no twisted angles}
\label{r0t}
\end{figure}

\begin{figure}
\centering
\includegraphics[trim = 0cm 8cm 0cm 2cm, clip=true, totalheight=4cm]{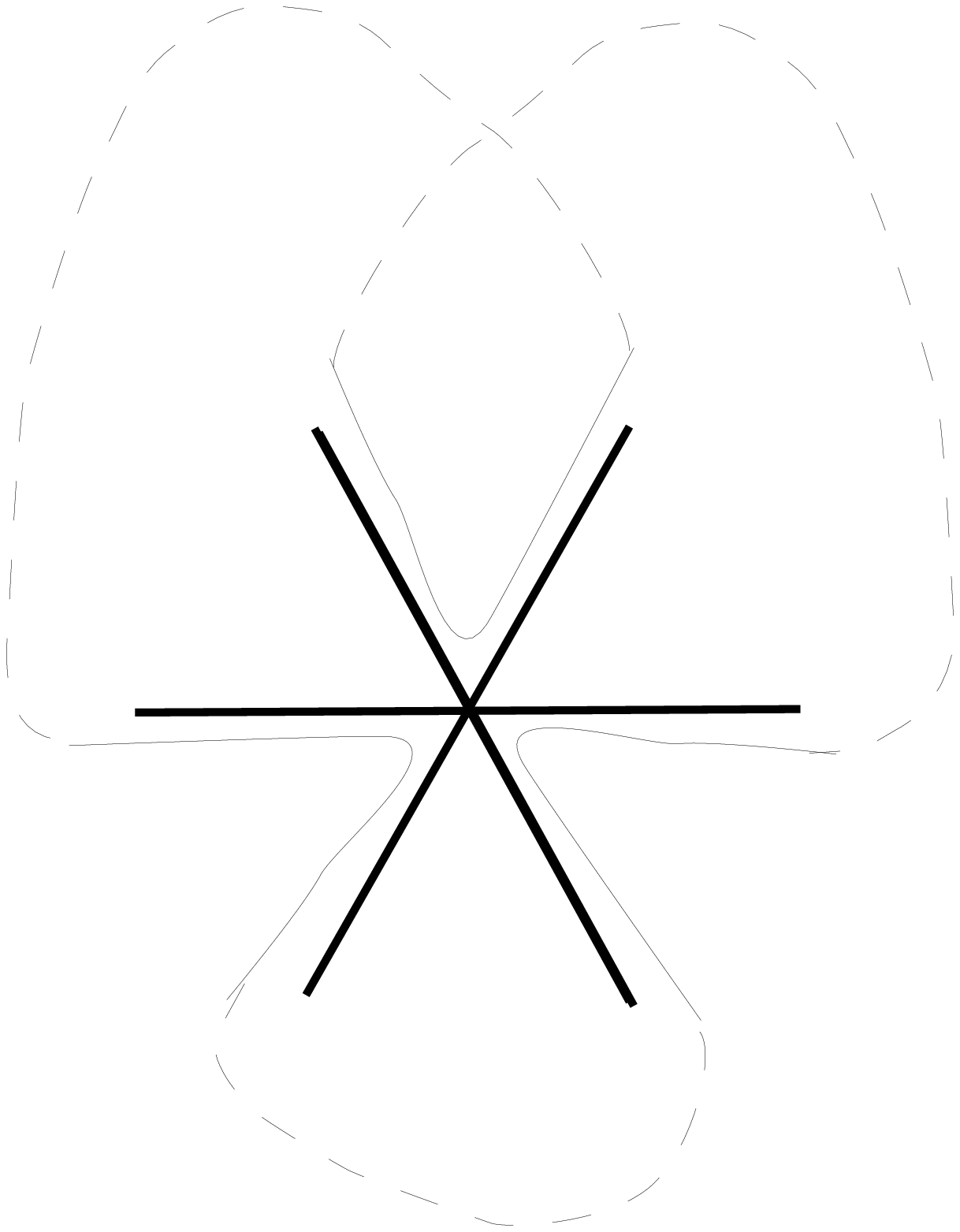}
\caption{A rotating vertex with one twisted angle}
\label{r1t}
\end{figure}

\begin{figure}
\centering
\includegraphics[trim = 0cm 8cm 0cm 1cm, clip=true, totalheight=4cm]{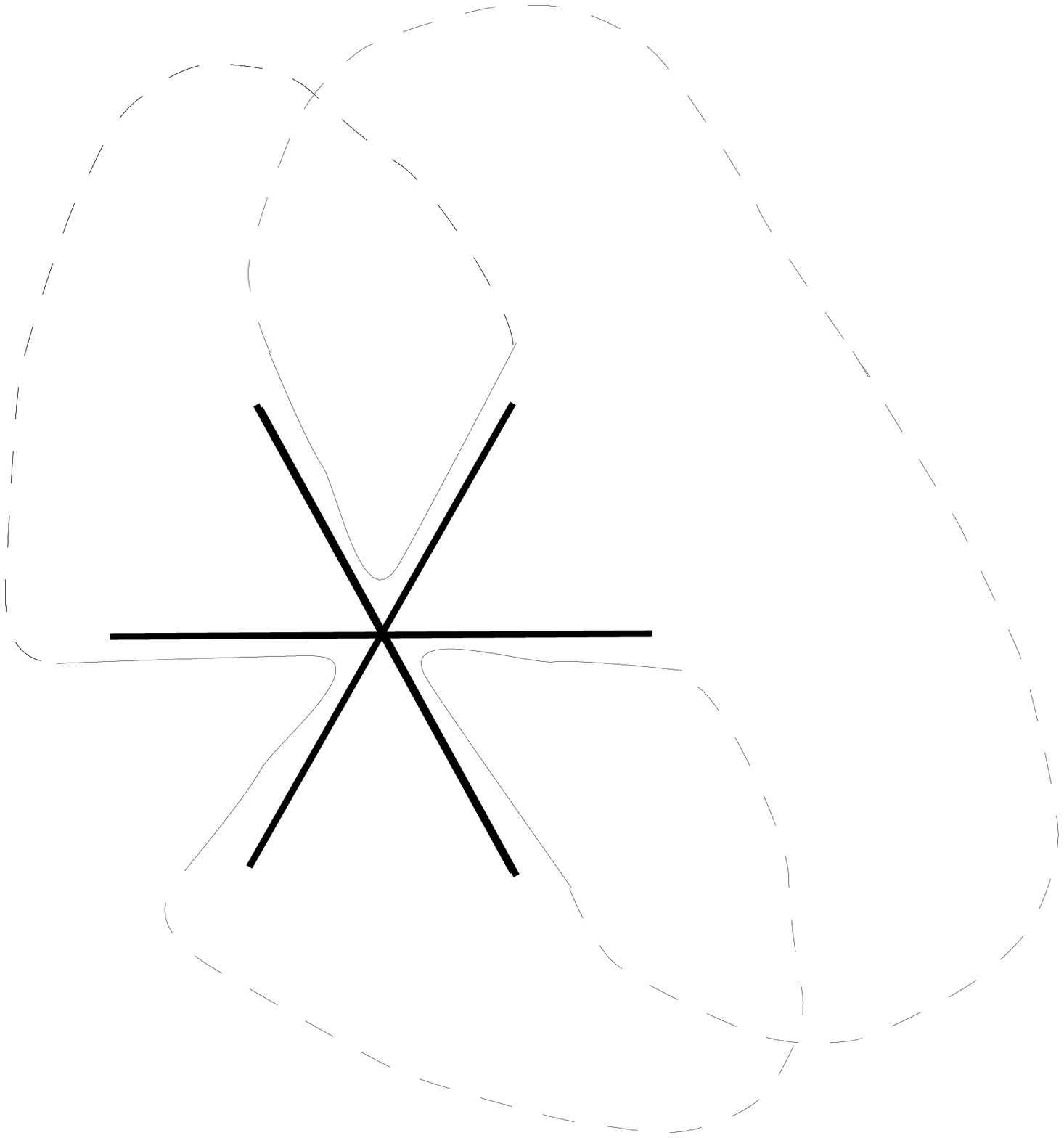}
\caption{A rotating vertex with two twisted angles}
\label{r2t}
\end{figure}

\begin{figure}
\centering
\includegraphics[trim = 0cm 6cm 0cm 2cm, clip=true, totalheight=4cm]{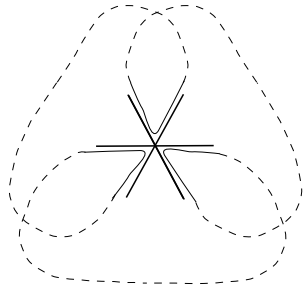}
\caption{A rotating vertex with three twisted angles}
\label{r3t}
\end{figure}

\begin{lemma}\label{rscycle-exists}
If $\Gamma$ is a connected $\st$-graph in which all vertices have degree 4 or 6, then $\Gamma$ admits a rotating-splitting circuit.
\end{lemma}

\begin{proof}
Assign to each vertex any rotating or splitting structure. This gives a partition of the edges of $\Gamma$ into edgewise disjoint cycles. If there is only one such cycle, we are done. If there is more than one cycle, since the graph is connected, there must be a vertex $v$ shared by different cycles. If $v$ has degree 4, it must be rotating, and we can join the two cycles meeting at $v$ by assigning to $v$ the other possible rotating structure. If $v$ has degree 6, we consider the cycles given by starting at $v$, exiting through one of its incident edges, and following the rotating-splitting structure until we come back to $v$. There are three such cycles, up to a change in orientation. If each of these cycles contains a pair of adjacent edges at $v$, we can assign to $v$ the rotating structure shown on the right side of Figure \ref{r0t-intro}, so that the three cycles are joined together, and $v$ becomes a rotating vertex with no twisted angles. Note that before making the change, $v$ may have some structure other than that shown on the left side of Figure $\ref{r0t-intro}$; the left side of the figure and the others referenced in this proof are merely examples. If exactly two of the three cycles contain a pair of adjacent edges around $v$, then the third must contain a pair of opposite edges, and we can assign to $v$ the splitting structure shown on the right side of Figure \ref{s0t-intro} to join the cycles together. If exactly one of the three cycles contains a pair of adjacent edges around $v$, then the other two must contain pairs of edges which are neither opposite nor adjacent. In this case we can join the cycles by assigning the rotating structure shown on the right side of Figure \ref{r1t-intro}, so that $v$ becomes a rotating vertex with one twisted angle. If none of the cycles contains a pair of adjacent edges, then we have two possibilities: Each of the cycles contains a pair of opposite edges, or one of the cycles contains a pair of opposite edges and the other two contain a pair of edges which are neither opposite nor adjacent. If each of the cycles contains a pair of opposite edges, we can assign to $v$ the rotating structure shown on the right side of Figure \ref{r3t-intro} to join the cycles. If one of the cycles contains a pair of opposite edges and the other contains a pair of edges which are neither opposite nor adjacent, we can assign to $v$ the rotating structure shown on the right side of Figure \ref{r2t-intro} to join the cycles.
\end{proof}

\begin{figure}
\centering
\includegraphics[trim = 0cm 19cm 0cm 1cm, clip = true, totalheight = 4cm]{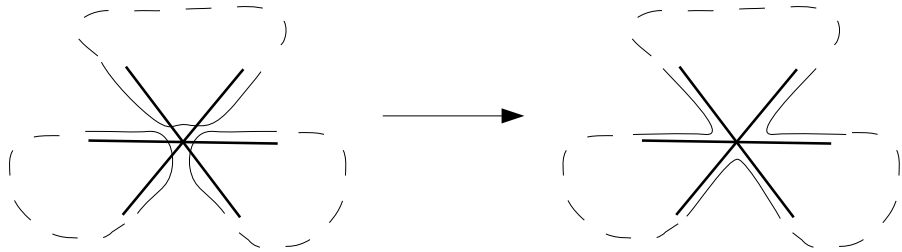}
\caption{Decreasing the number of cycles by introducing a rotating structure with no twisted edges at a vertex, as referenced in Lemma \ref{rscycle-exists}}.
\label{r0t-intro}
\end{figure}

\begin{figure}
\centering
\includegraphics[trim = 0cm 17cm 0cm 3cm, clip = true, totalheight = 4cm]{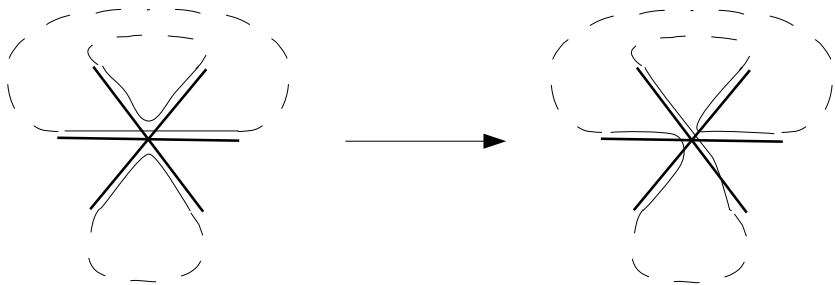}
\caption{Decreasing the number of cycles by introducing a splitting structure at a vertex, as referenced in Lemma \ref{rscycle-exists}}.
\label{s0t-intro}
\end{figure}

\begin{figure}
\centering
\includegraphics[trim = 2cm 17cm 0cm 2cm, clip = true, totalheight = 4cm]{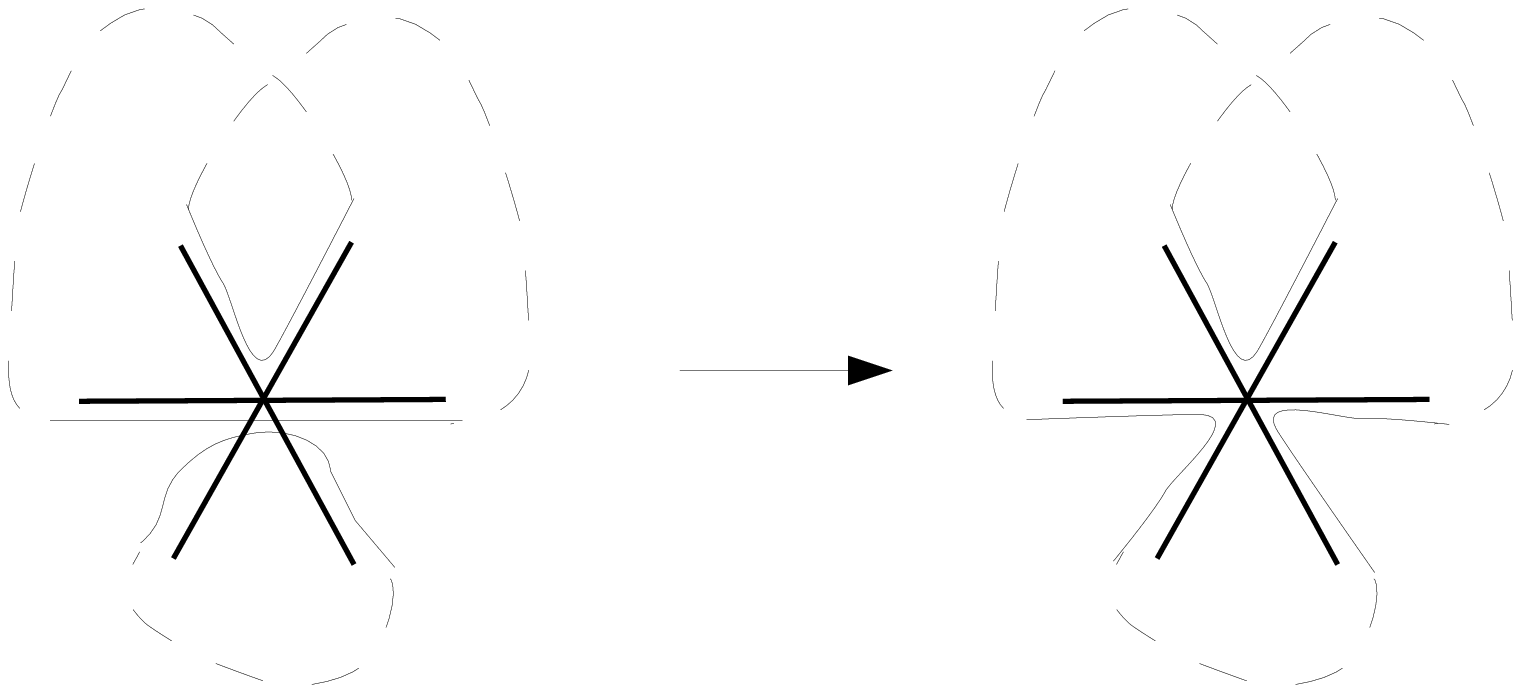}
\caption{Decreasing the number of cycles by introducing a rotating structure with one twisted edge at a vertex, as referenced in Lemma \ref{rscycle-exists}}.
\label{r1t-intro}
\end{figure}

\begin{figure}
\centering
\includegraphics[trim = 0cm 15cm 0cm 5cm, clip = true, totalheight = 4cm]{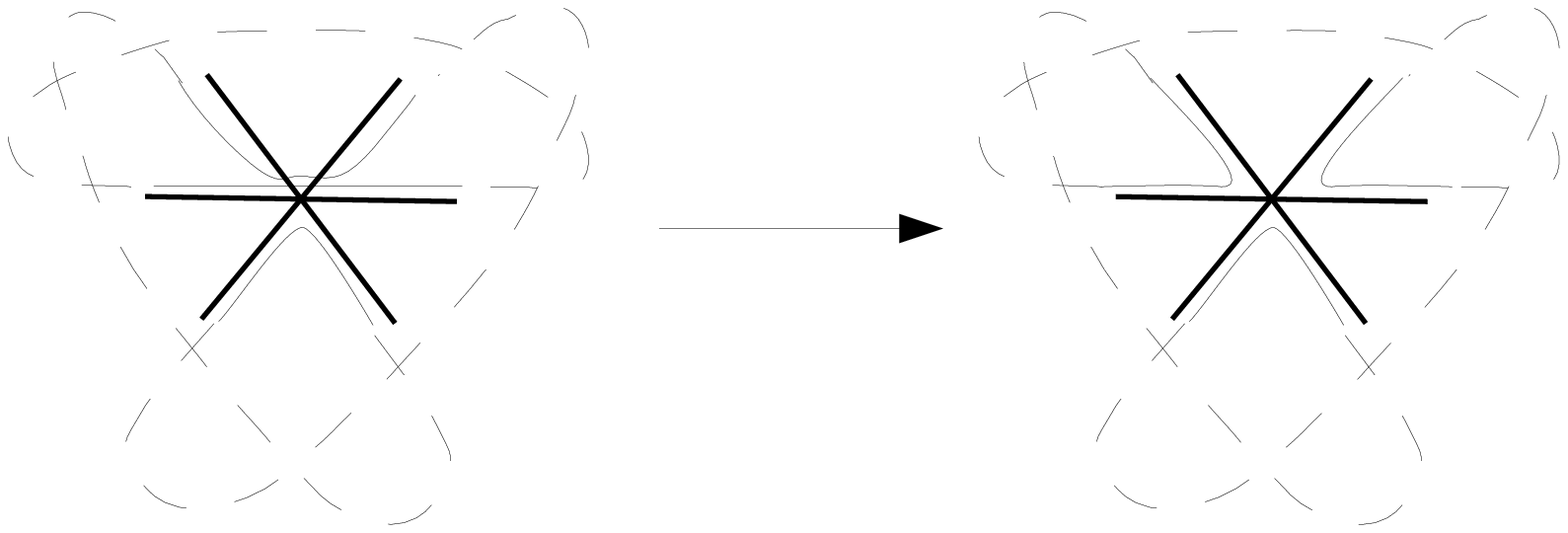}
\caption{Decreasing the number of cycles by introducing a rotating structure with three twisted edges at a vertex, as referenced in Lemma \ref{rscycle-exists}}.
\label{r3t-intro}
\end{figure}

\begin{figure}
\centering
\includegraphics[trim = 0cm 17cm 0cm 1.6cm, clip = true, totalheight = 4cm]{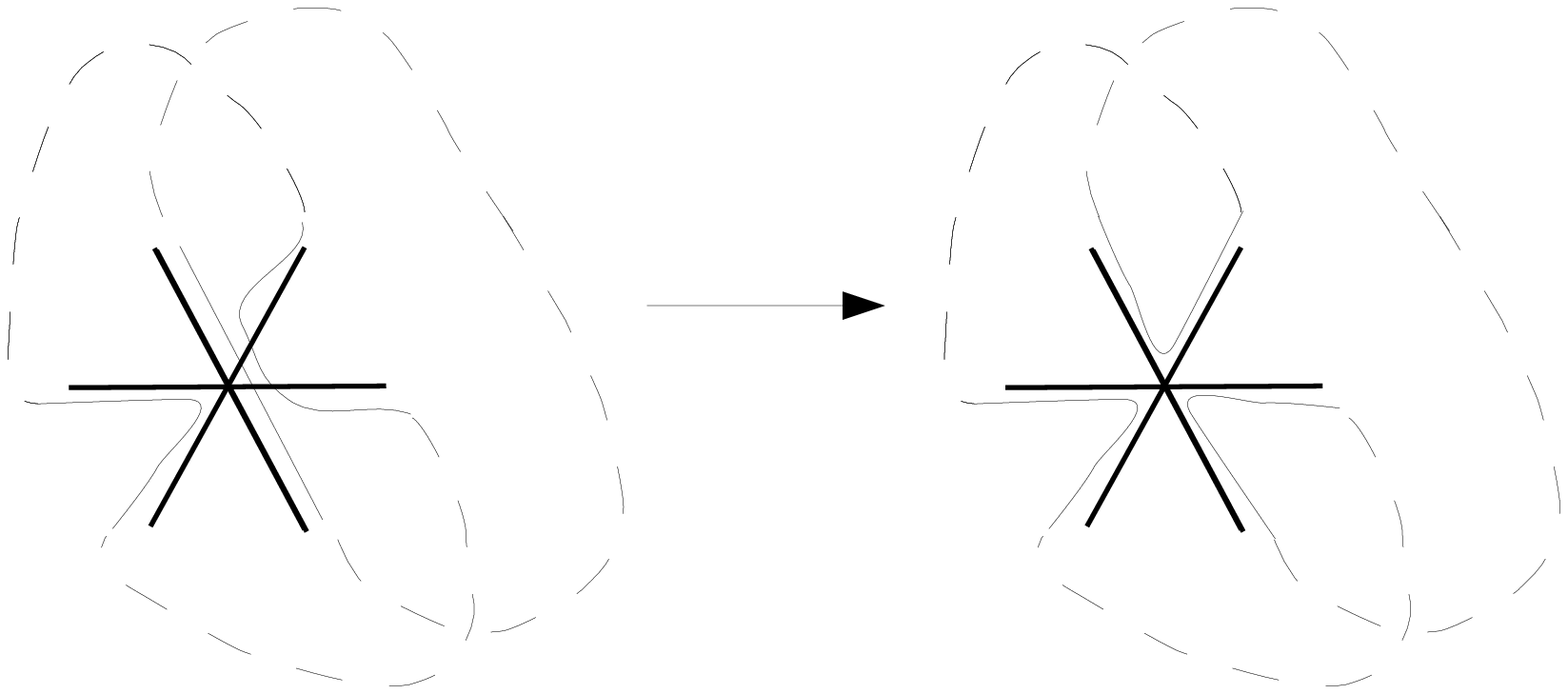}
\caption{Decreasing the number of cycles by introducing a rotating structure with two twisted edges at a vertex, as referenced in Lemma \ref{rscycle-exists}}.
\label{r2t-intro}
\end{figure}

%\begin{figure}
%\centering
%\includegraphics[trim = 0cm 6cm 0cm 2cm, clip=true, totalheight=4cm]{s0t.eps}
%\caption{}
%\label{s0t}
%\end{figure}

%[Here I need to classify the vertices in more detail. There are ultimately two types of 4-vertices (positive and negative), four types of rotating 6-vertices, and three types of splitting 6-vertices. The classification of rotating 6-vertices poses the greatest challenge. It would probably be possible to classify them into the following classes, which you seemed to be suggesting: Flat without a twisted angle, flat with a twisted angle, crossed without a twisted angle, crossed with a twisted angle. However, this would require us to redefine flat and crossed, and I think the new definitions would be fairly messy. I think it would be better to take advantage of the fact that "crossed" can be thought of as "flat with three twisted angles", so the classes can be: zero twisted angles, one twisted angle, two twisted angles, three twisted angles. Figures below.]

\begin{definition}
A \emph{chord diagram} is a cubic graph $D$ with a distinguished Hamiltonian cycle; i.e. an embedding $S^1 \to D$ which covers all the vertices of $D$.
\end{definition}

\begin{definition}
A \emph{signed chord diagram} is a chord diagram in which each edge not in the distinguished cycle is assigned a positive or negative sign.
\end{definition}

\begin{definition}
A $\st$-chord diagram is a graph $D$ with a distinguished simple cycle (i.e. an embedding $S^1\to D$), such that every vertex in $D$ has degree 3 or 4 and for every edge $e$ in $D$, one of the following holds:
\begin{enumerate}
\item $e$ is in the distinguished cycle.
\item Both of the vertices on $e$ are in the distinguished cycle, and both have degree 3.
\item One of the vertices on $e$ is in the distinguished cycle, the other is not, and both have degree 3.
\item Both of the vertices on $e$ are in the distinguished cycle, one has degree 3, and the other has degree 4.
\end{enumerate}
\end{definition}

\begin{definition}
A \emph{signed $\st$-chord diagram} is a $\st$-chord diagram in which each edge not in the distinguished cycle is assigned a positive or negative sign.
\end{definition}

\begin{definition}
An \emph{arc} of a $\st$-chord diagram $\Gamma$ is an edge in the distinguished cycle of $\Gamma$.
\end{definition}

\begin{definition}
A \emph{chord} of a $\st$-chord diagram $\Gamma$ is an edge not in the distinguished cycle of $\Gamma$, connecting two vertices of degree 3 which are in the cycle.
\end{definition}

\begin{definition}
A \emph{triad} of a $\st$-chord diagram $\Gamma$ is a vertex $v$ not in the distinguished cycle of $\Gamma$, together with the three edges incident to $v$. The vertex $v$ is called a \emph{triad point}.
\end{definition}

\begin{definition}
A \emph{double chord} of a $\st$-chord diagram $\Gamma$ is a pair of edges not in the distinguished cycle of $\Gamma$, which are incident to a shared vertex $v$. The vertex $v$ is called the \emph{principal vertex} of the double chord.
\end{definition}

Given a $\st$-graph $\Gamma$ with all vertices of degree 4 or 6, and given a rotating-splitting circuit $C$ of $\Gamma$, we define a $\st$-chord diagram $D_{\Gamma, C}$ as follows:

For each 4-vertex $v$ in $\Gamma$, the two points in $S^1$ which are mapped to $v$ by $C$ are connected by a chord, whose sign is positive if and only if the two half-edges through which $C$ enters $v$ are not adjacent. For any rotating 6-vertex $v$ in $\Gamma$, the three points in $S^1$ which are mapped to $v$ by $C$ are connected by a triad, and the edge connecting a vertex $a \in C^{-1}(\{v\})$ to the triad point has positive sign if and only the angle into which $C$ maps a neighborhood of $a$ is not twisted. For any splitting 6-vertex $v$ in $\Gamma$, the three points in $S^1$ which are mapped to $C$ are connected by a double chord, whose principal vertex is $a \in C^{-1}(\{v\})$ such that $C(a-\epsilon)$ and $C(a+\epsilon)$ are in opposite half-edges around $v$, where the sign of the edge connecting $a$ another vertex $b \in C^{-1}(\{v\})$ is positive if and only if $C(b+\epsilon)$ is adjacent to $C(a-\epsilon)$ [Figure].

\begin{definition}
An \emph{expansion} of as signed $\st$-chord diagram $D_{\Gamma, C}$ a signed chord diagram $D'_{\Gamma, C}$ such that
\begin{enumerate}
\item For every chord in $D,{\Gamma,C}$ containing vertices $a$ and $b$, there is a chord of the same sign in $D'_{\Gamma,C}$ of the same sign connecting vertices $a$ and $b$.
\item For every triad $e$ in $D_{\Gamma,C}$ containing at least one edge with positive sign, for some labeling $a,b,c$ of the vertices of $e$ such that the edge connecting $a$ to the triad point is positive, $D'_{\Gamma,C}$ contains a chord connecting $a \pm \epsilon$ to $b$ and a chord connecting $a \mp \epsilon$ to $c$, with $\pm \epsilon$ chosen in such a way that the chords are not linked. Furthermore, the chords connecting $a$ to $b$ and $a$ to $c$ in $D'_{\Gamma,C}$ have the same signs as the edges connecting the triad point of $e$ to $b$ and $c$ in $D_{\Gamma,C}$, respectively.
\item For every triad $e$ in $D_{\Gamma,C}$ in which all edges have negative sign, for some labeling $a,b,c$ of the vertices of $e$, $D'_{\Gamma,C}$ contains a chord connecting $a \pm \epsilon$ to $b$ and a chord connecting $a \mp \epsilon$ to $c$, with $\pm \epsilon$ chosen in such a way that the chords are linked. Furthermore, the chords connecting $a$ to $b$ and $a$ to $c$ in $D'_{\Gamma,C}$ have signs opposite to the edges connecting the triad point of $e$ to $b$ and $c$ in $D_{\Gamma,C}$, respectively.
\item For every double chord $e$ in $D_{\Gamma,C}$ with principal vertex $a$, for some labeling $b, c$ of the nonprincipal vertices of $e$, there is a chord in $D'_{\Gamma,C}$ connecting $a - \epsilon$ to $b$ and a chord in $D'_{\Gamma,C}$ connecting $a + \epsilon$ to $c$. These chords are not linked.
\end{enumerate}
\end{definition}

\begin{definition}
A \emph{permissible separation} of a signed chord diagram $D'_{\Gamma, C}$ arising as an expansion of a signed $\st$-chord diagram $D_{\Gamma, C}$ is a pair of signed chord diagrams $D_W$ and $D_B$ such that
\begin{enumerate}
\item Every chord in $D'_{\Gamma, C}$ is in exactly one of $D_W$ and $D_B$.
\item Two chords in $D'_{\Gamma, C}$ which come from the same triad in $D_{\Gamma, C}$ are both in $D_W$, or both in $D_B$.
\item Of any two chords in $D'_{\Gamma, C}$ which come from the same double chord in $D_{\Gamma, C}$, one chord is in $D_W$ and the other is in $D_B$.
\end{enumerate}
\end{definition}

Suppose $\Gamma$ is checkerboard-embedded in a closed surface $S$. Then the rotating-splitting circuit $C$ gives a mapping from $S^1$ to $S$ which is one-to-one except at the preimages of vertices $\Gamma$. This mapping can be smoothed to give an embedding of $S^1$ into $S$, as in Figure \ref{fig:smoothed_rscycle}. Observe that the circle $S^1 \subset S$ divides the surface into a black part and a white part. We can draw the chords of $D'_{\Gamma, C}$ as small edges lying in neighborhoods of vertices of $\Gamma$, see Figure \ref{fig:small_chord}.

\begin{figure}
	\caption{Chords drawn as small edges at:}
	\begin{subfigure}{0.3\textwidth}
	\caption{a 4-vertex}
	\centering
	\includegraphics[trim = 8cm 20cm 0cm 3cm, clip=true, totalheight=3cm]{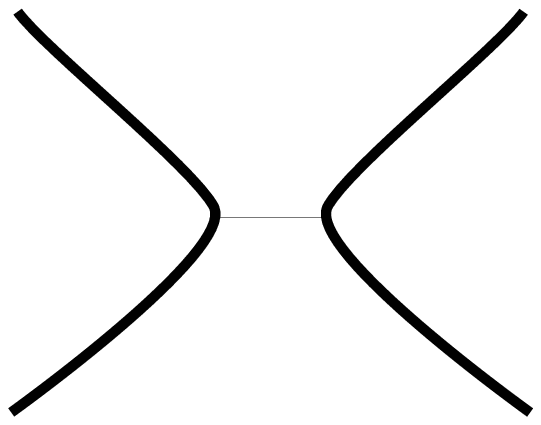}
	\end{subfigure}
	\quad
	\begin{subfigure}{0.3\textwidth}
	\caption{a rotating 6-vertex}
	\centering
	\includegraphics[trim = 5cm 17cm 4cm 5cm, clip=true, totalheight=3cm]{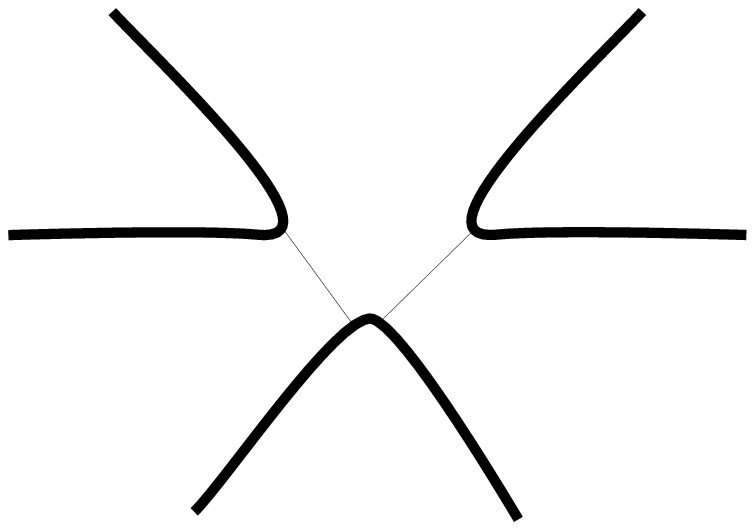}
	\end{subfigure}
		\quad
	\begin{subfigure}{0.3\textwidth}
	\caption{a splitting 6-vertex}
	\centering
	\includegraphics[trim = 6cm 17cm 4cm 4cm, clip=true, totalheight=3cm]{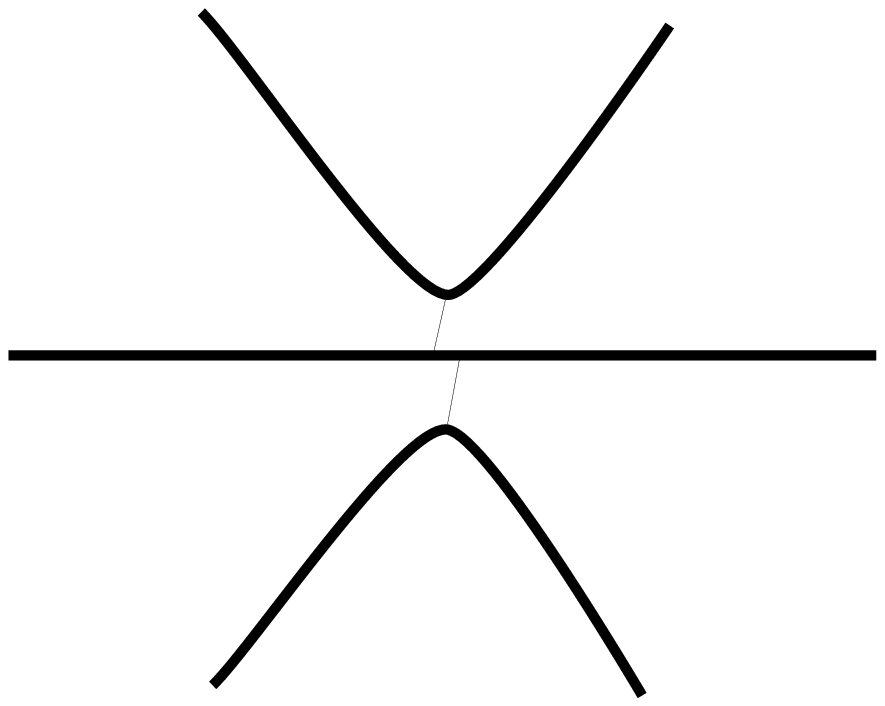}
	\end{subfigure}
	\label{fig:small_chord}
\end{figure}

The coloring of $S$ divides the chords of $D'_{\Gamma, C}$ into two families: those lying in the white part and those lying in the black part. Observe that the two chords in the neighborhood of a rotating vertex are in the same part, and the two chords in the neighborhood of a splitting vertex are in different parts. Thus we have a permissible separation of $D'_{\Gamma, C}$.

\begin{figure}
\centering
\includegraphics[trim = 2cm 7cm 0cm 7cm, clip=true, totalheight=6cm]{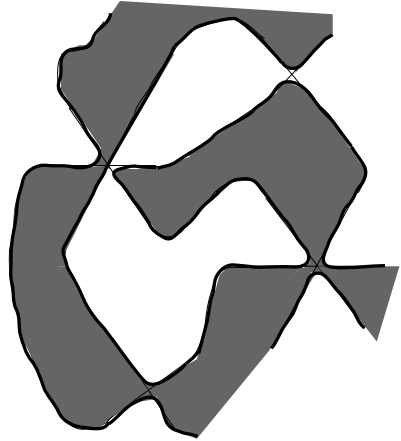}
\caption{The circle $S^1 \subset S$ divides the surface into a black part and a white part.}
\label{fig:smoothed_rscycle}
\end{figure}

Vice versa, given a $\st$-graph $\Gamma$ with all vertices of degree 4 or 6 and which satisfies the source-sink condition, a rotating-splitting circuit $C$ of $\Gamma$, and a permissible separation of $D'_{\Gamma, C}$, we can recover the coloring of the angles around each vertex of $\Gamma$, and thus we can recover the surface $S$. Thus given a $\st$-graph $\Gamma$ with all vertices of degree 4 or 6 and which satisfies the source-sink condition, and an expansion $D'_{\Gamma, C}$ of its $\st$-chord diagram, we have a one-to-one correspondence between atoms of $\Gamma$ and permissible separations of $D'_{\Gamma, C}$.

Note that the two chords to be drawn in the neighborhood of any rotating 6-vertex $v$ do not cross in $S$, as shown in Figures \ref{r0t-lifecycle}, \ref{r1t-lifecycle}, \ref{r2t-lifecycle}, and \ref{r3t-lifecycle}. Thus we have an embedding of $D'_{\Gamma, C}$ into $S$. Furthermore, since the embedding of $\Gamma$ divides $S$ into 2-cells, the embedding of $D'_{\Gamma, C}$ does as well.

\begin{figure}
\centering
\includegraphics[trim = 0cm 20cm 0cm 1.5cm, clip=true, totalheight=4cm]{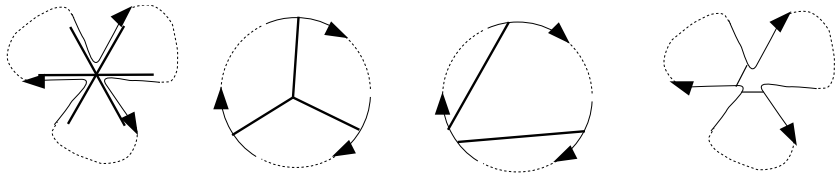}
\caption{The chords drawn in the neighborhood of a rotating 6-vertex with no twisted edges do not cross.}
\label{r0t-lifecycle}
\end{figure}

\begin{figure}
\centering
\includegraphics[trim = 0cm 21cm 0cm 1cm, clip=true, totalheight=4cm]{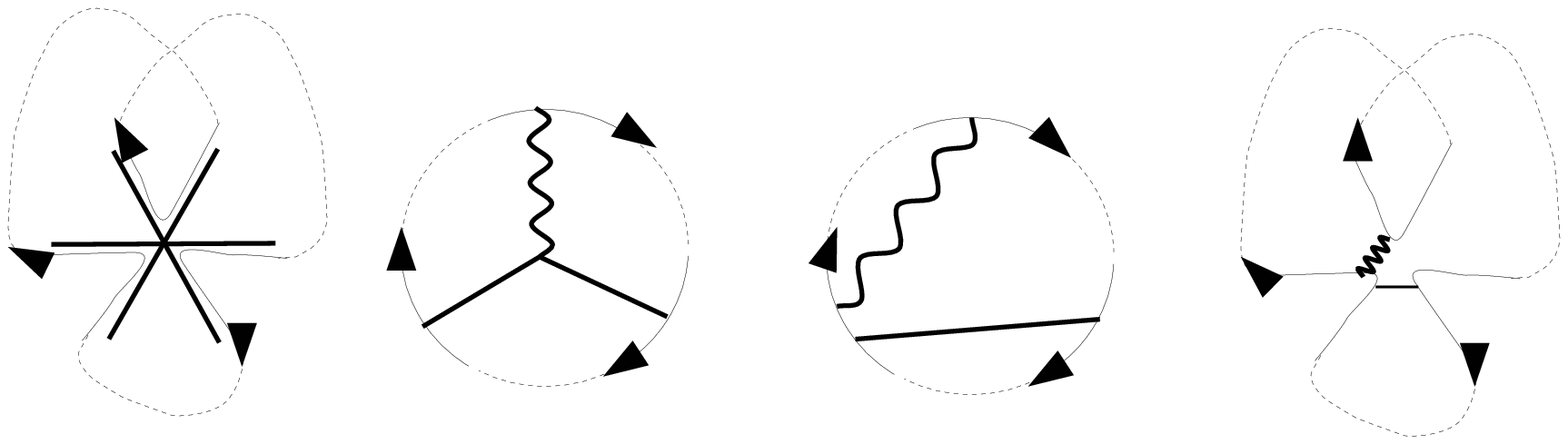}
\caption{The chords drawn in the neighborhood of a rotating 6-vertex with one twisted edge do not cross.}
\label{r1t-lifecycle}
\end{figure}

\begin{figure}
\centering
\includegraphics[trim =2cm 21cm 0cm 1cm, clip=true, totalheight=5cm]{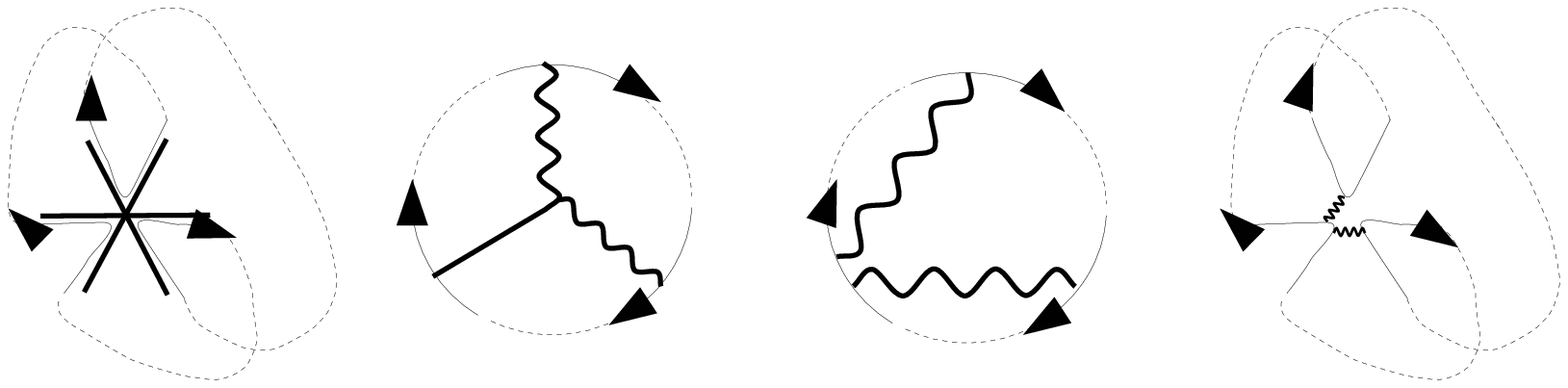}
\caption{The chords drawn in the neighborhood of a rotating 6-vertex with two twisted edges do not cross.}
\label{r2t-lifecycle}
\end{figure}

\begin{figure}
\centering
\includegraphics[trim = 3cm 22cm 0cm 1.5cm, clip=true, totalheight=4cm]{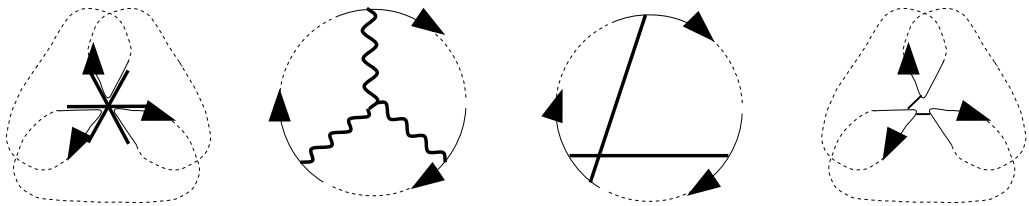}
\caption{The chords drawn in the neighborhood of a rotating 6-vertex with three twisted edges do not cross.}
\label{r3t-lifecycle}
\end{figure}

\begin{definition}
Given a chord diagram $D$, \emph{surgery} of $D$ is the following process: For each chord $e$ connecting points $a,b$, delete a neighborhood of $e$ and connect the obtained endpoints $a+\epsilon$ to $b-\epsilon$ and $a-\epsilon$ to $b+\epsilon$ if $e$ is positive, and $a+\epsilon$ to $b+\epsilon$ and $a-\epsilon$ to $b-\epsilon$ if $e$ is negative. This produces a family of circles; these are called the \emph{result of surgery} of $D$.
\end{definition}

\begin{definition}
To form the \emph{intersection matrix} of a signed chord diagram $D$ with $n$ chords, first enumerate the chords $1,\ldots,n$. Then the \emph{intersection matrix} $M(D)$ is an $n \times n$ matrix over $\mathbb Z_2$, such that $M_{ii}=1$ if and only if the chord $i$ is negative, and $M_{ij}=M_{ji}=1$ for $i \neq j$ if and only if the chords $i$ and $j$ are linked.
\end{definition}

\begin{theorem}[Circuit-Nullity Theorem \cite{cohn},\cite{soboleva},\cite{traldi}]
The number of components in the manifold obtained from a signed chord diagram $D$ by surgery of the circle is one plus the corank of $M(D)$.
\end{theorem}

\begin{lemma}\label{sep-to-genus}
Given a $\st$-graph $\Gamma$ in which all vertices have degree 4 or 6 and a rotating-splitting circuit $C$ of $\Gamma$, and a checkerboard embedding of $\Gamma$ into a nonorientable surface $S$, the nonorientable genus of $S$ is given by
\[\mbox{rank}(M(D_W))+\mbox{rank}(M(D_B))\]
where $D_W$ and $D_B$ are the results of the permissible separation of $D'_{\Gamma, C}$ induced by the embedding.
\end{lemma}
\begin{proof}
Consider the embedding $D'_{\Gamma, C}\to S$ described above.
The number of 2-cells on the white side of the embedding is the number of circles resulting in surgery of $D_W$. Likewise, the number of 2-cells on the black side of the embedding is the number of circles resulting in surgery of $D_B$. Applying the Circuit Nullity Theorem, the total number of 2-cells is $\mbox{corank}(M(D_W))+\mbox{corank}(M(D_B))+2$. Introducing the notation $|D'_{\Gamma, C}|$ to represent the number of chords in $D'_{\Gamma, C}$, the number of arcs in $D'_{\Gamma, C}$ is $2|D'_{\Gamma, C}|$, so its total number of edges is $3|D'_{\Gamma, C}|$. The number of vertices in $D'_{\Gamma, C}$ is $2|D'_{\Gamma, C}|$. Thus the Euler characteristic of $S$ is
\[ 2|D'_{\Gamma, C}| - 3|D'_{\Gamma, C}| + \mbox{corank}(M(D_W))+\mbox{corank}(M(D_B))+2\]
\[= -\mbox{rank}(M(D_W))-\mbox{rank}(M(D_B))+2 \]
so the nonorientable genus of $S$ is $\mbox{rank}(M(D_W))+\mbox{rank}(M(D_B))$.
\end{proof}

Thus $\Gamma$ has admits an atom of genus $g$ if and only if some permissible separation of $D_{\Gamma, C}$ results in $\mbox{rank}(D_W)+\mbox{rank}(D_B)=g$. This can be reduced to a problem on matrices, as follows.

\begin{definition}
A \emph{permissible partition} of the indices of $M(D'_{\Gamma, C})$ is a partition of the indices of $M(D'_{\Gamma, C})$ (which are just the chords of $D'_{\Gamma, C}$) into two parts, in such a way that chords arising from the same triad in $D_{\Gamma, C}$ are in the same part, and chords arising from the same double chord in $D_{\Gamma, C}$ are in different parts.
\end{definition}

Clearly, $D_W$ and $D_B$ are a permissible separation of $D'_{\Gamma, C}$ if and only if there exists a permissible partition $I \sqcup J$ of the indices of $M(D'_{\Gamma, C})$ such that $M(D_{\Gamma, C})_I$ is the intersection matrix of $D_W$ and $M(D_{\Gamma, C})_J$ is the intersection matrix of $D_B$.

\section{Main Result}

\begin{theorem}\label{main}
For a $\st$-graph $\Gamma$ which does not satisfy the source-target condition and which has rotating-splitting circuit $C$, $\Gamma$ has a checkerboard embedding into a nonorientable surface of genus $g$ if and only if there is a permissible partition of the indices of $M(D_{\Gamma, C})$ into parts $I$ and $J$ such that $\mbox{rank}(M_I)+\mbox{rank}(M_J)=g$.
\end{theorem}
\begin{proof}
Let $D'_{\Gamma, C}$ be any expansion of $D_{\Gamma, C}$.
By Lemma \ref{sep-to-genus}, $\Gamma$ has a checkerboard embedding into a surface of genus $g$ if and only there is a permissible separation $(D_W,D_B)$ of $D'_{\Gamma, C}$ such that $\mbox{rank}(M(D_W)) + \mbox{rank}(M(D_B))=g$. Such a permissible separation exists if and only if there is a permissible partition of $M(D_{\Gamma, C})$ into parts $I$ and $J$ such that $\mbox{rank}(M_I)+\mbox{rank}(M_J)=g$.
\end{proof}

Thus, the problem of finding the minimal nonorientable genus into which a $\st$-chord diagram with each vertex of degree 4 or 6 may be checkerboard-embedded, is equivalent to the problem of finding a permissible partition of the indices of a matrix $M$ into parts $I$ and $J$ which minimizes $\mbox{rank}(M_I)+\mbox{rank}(M_J)$.

\section{The case of $\mathbb RP^2$}
A $\st$-graph $\Gamma$ with rotating-splitting circuit $C$ is embeddable into the projective plane if and only if there exists a permissible separation $(D_W,D_B)$ of $D'_{\Gamma,C}$ such that $\mbox{rank}(M(D_W))=1$ and $\mbox{rank}(M(D_B))=0$. In other words, $\Gamma$ is $\mathbb RP^2$-embeddable if and only if there exists a permissible separation of $D'_{\Gamma, C}$ into two chord diagrams, one of which consists of a family of pairwise-linked negative chords and a family of positive chords which are not linked to each other or to the negative chords, and the other of which consists of a family of pairwise-unlinked negative chords. We can test this condition by the following algorithm, which takes time quadratic in the number of chords of $D'_{\Gamma,C}$: First assign all negative chords to the same chord diagram. Then for each assigned chord, assign all positive chords linked to it to the other chord diagram. If an assigned chord originates from a triad, assign the other chord coming from this triad to the same chord diagram, and if the assigned chord originates from a double chord, assign the other chord coming from this double chord to the other chord diagram. Then, for each of the newly assigned chords, assign any unassigned linked chords or chords coming from the same triad or double chord, using the same rules described above. Repeat this process until for every assigned chord, the linked chords and any chord coming from the same triad or double chord have been assigned. If not all chords have been assigned, take any unassigned chord and arbitrarily assign it to $D_W$ or $D_B$, and repeat until all chords have been assigned. Finally, check whether this is a permissible separation, and whether $\mbox{rank}(M(D_W))+\mbox{rank}(M(D_B))=1$. $\Gamma$ is $\mathbb RP^2$-embeddable if and only if both of these conditions are true.

\section{The case of the Klein bottle}
A $\st$-graph $\Gamma$ with rotating-splitting circuit $C$ is embeddable into the Klein bottle if and only if there exists a permissible separation $(D_W,D_B)$ of $D'_{\Gamma,C}$ such that $\mbox{rank}(M(D_W))+\mbox{rank}(M(D_B))=2$. There are two possible cases in which this can occur: $\mbox{rank}(M(D_W))=\mbox{rank}(M(D_B))=1$ or $\mbox{rank}(M(D_W))=2$ and $\mbox{rank}(M(D_B))=0$.

\begin{figure}
\centering
\includegraphics[trim = 0cm 18cm 0cm 2cm, clip=true, totalheight=4cm]{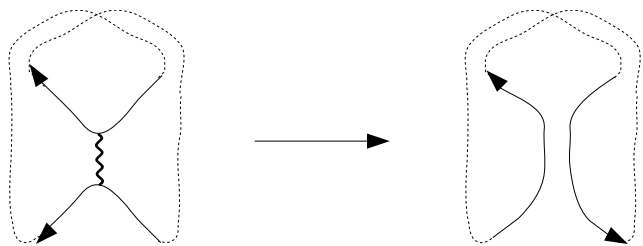}
\caption{Surgery at a negative chord, for determining embeddability into the Klein bottle}
\label{klein-surgery}
\end{figure}

We will first consider the case where $\mbox{rank}(M(D_W))=\mbox{rank}(M(D_B))=1$. In this case, we have a permissible separation of $D'_{\Gamma, C}$, each of which consists of a family of pairwise-linked negative chords and a family of positive chords which are not linked to each other or to the negative chords. This condition also admits a quadratic-time test, as follows: Assign one of the chords arbitrarily to $D_W$ or $D_B$. If the chord is positive, assign all chords linked to it to the other chord diagram; if it is negative, assign all positive linked chords and all negative unlinked chords to the other diagram. Regardless of sign, if an assigned chord originates from a triad, assign the other chord coming from this triad to the same chord diagram, and if the assigned chord originates from a double chord, assign the other chord coming from this double chord to the other chord diagram. Repeat this process until for every assigned chord, the linked chords and any chord coming from the same triad or double chord have been assigned. If not all chords have been assigned, take any unassigned chord and arbitrarily assign it to $D_W$ or $D_B$, and repeat until all chords have been assigned. Finally, check whether this is a permissible separation, and whether $\mbox{rank}(M(D_W))=\mbox{rank}(M(D_B))=1$. These conditions are met if and only if there is an embedding of $\Gamma$ into the Klein bottle so that a smoothing of $C$ divides the Klein bottle into two M\"obius bands.

If this test fails, there is still the possibility that $\Gamma$ has an embedding into the Klein bottle where the smoothing of $C$ bounds a disc.
To cover this possibility, we choose a negative chord $c$ of $D'_{\Gamma,C}$ and perform surgery at that chord, in the manner shown in \ref{klein-surgery}. Since this reverses the orientation of part of the designated cycle, we should also change the sign of all chords which cross $c$, producing a new chord diagram $D''_{\Gamma_C}$. Then for any surface $S$ and any embedding of $D'_{\Gamma,C} \to S$ which respects the signs of the chords, there is a corresponding embedding $D''_{\Gamma,C} \to S$, still respecting the signs of the chords. Furthermore, if the distinguished cycle in the embedding $D'_{\Gamma,C} \to S$ into the Klein bottle bounds a disc, then the distinguished cycle in the embedding $D''_{\Gamma,C} \to S$ bounds a M\"obius band. Thus $\Gamma$ has an embedding into the Klein bottle where the smoothing of $C$ bounds a disc if and only $D''_{\Gamma,C}$ has an embedding into the Klein bottle where the distinguished cycle bounds a M\"obius band. This condition can be checked using the algorithm in the previous paragraph.

\end{document}